\documentclass[a4paper,11pt]{article}
\usepackage[latin1]{inputenc}
\usepackage[english]{babel}
\usepackage{amsmath}
\usepackage{amsfonts}
\usepackage{amssymb}
\usepackage{epsfig}
\usepackage{amsopn}
\usepackage{amsthm}
\usepackage{color}
\usepackage{graphicx}
\usepackage{enumerate}
\newtheorem{theorem}{Theorem}[section]
\newtheorem{corollary}[theorem]{Corollary}
\newtheorem{lemma}[theorem]{Lemma}
\newtheorem{proposition}[theorem]{Proposition}
\newtheorem{definition}[theorem]{Definition}
\newtheorem{remark}[theorem]{Remark}
\newtheorem*{theorem*}{Theorem}
\newtheorem*{lemma*}{Lemma}
\newtheorem*{remark*}{Remark}
\newtheorem*{definition*}{Definition}
\newtheorem*{proposition*}{Proposition}
\newtheorem*{corollary*}{Corollary}
\numberwithin{equation}{section}
\newcommand{\real}{\mathbb{R}}

\let\ced=\c         

\def\e{\varepsilon}        

\def\t{\tau}

\def\cd{{\cal D}}

\def\cf{{\cal F}}

\def\qed{\,\unskip\kern 6pt \penalty 500
\raise -2pt\hbox{\vrule \vbox to8pt{\hrule width 6pt
\vfill\hrule}\vrule}\par}
\definecolor{darkblue}{rgb}{0.05, .05, .65}
\definecolor{darkgreen}{rgb}{0.1, .65, .1}
\definecolor{darkred}{rgb}{0.8,0,0}
\newcommand{\beqn}{\begin{equation}}
\newcommand{\eeqn}{\end{equation}}
\newcommand{\bear}{\begin{eqnarray}}
\newcommand{\eear}{\end{eqnarray}}
\newcommand{\bean}{\begin{eqnarray*}}
\newcommand{\eean}{\end{eqnarray*}}
\begin{document}

\title{\huge \bf Large time behavior for a porous medium equation in a nonhomogeneous medium with critical density}

\author{
\Large Razvan Gabriel Iagar\,\footnote{Dept. de An\'{a}lisis
Matem\'{a}tico, Universitat de Valencia, Dr. Moliner 50, 46100,
Burjassot (Valencia), Spain, \textit{e-mail:} razvan.iagar@uv.es},
\footnote{Institute of Mathematics of the Romanian Academy, P.O. Box
1-764, RO-014700, Bucharest, Romania.}
\\[4pt] \Large Ariel S\'{a}nchez,\footnote{Departamento de Matem\'{a}tica Aplicada,
Universidad Rey Juan Carlos, M\'{o}stoles, 28933, Madrid, Spain,
\textit{e-mail:} ariel.sanchez@urjc.es}\\ [4pt] }
\date{}
\maketitle

\begin{abstract}
We study the large time behavior of solutions to the porous medium
equation in nonhomogeneous media with critical singular density
$$
|x|^{-2}\partial_{t}u=\Delta u^m, \quad \hbox{in} \
\real^N\times(0,\infty),
$$
where $m>1$ and $N\geq3$. The asymptotic behavior proves to have
some interesting and striking properties. We show that there are
different asymptotic profiles for the solutions, depending on
whether the continuous initial data $u_0$ vanishes at $x=0$ or not.
Moreover, when $u_0(0)=0$, we show the convergence towards a profile
presenting a discontinuity in form of a shockwave, coming from an
unexpected asymptotic simplification to a conservation law, while
when $u_0(0)>0$, the limit profile remains continuous. These
phenomena illustrate the strong effect of the singularity at $x=0$.
We improve the time scale of the convergence in sets avoiding the
singularity. On the way, we also study the large-time behavior for a
porous medium equation with convection which is interesting for
itself.
\end{abstract}

\


\smallskip


\section{Introduction}

The goal of this paper is to study the asymptotic behavior of
solutions to the following porous medium equation in nonhomogeneous
media with critical singular density:
\begin{equation}\label{eq1}
|x|^{-2}\partial_{t}u(x,t)=\Delta u^m(x,t), \quad
(x,t)\in\real^{N}\times(0,\infty),
\end{equation}
with $m>1$ and $N\geq3$. An important feature of this equation is
the influence of the density that is at the same time singular at
$x=0$ and degenerate at infinity, giving rise to very interesting
and unexpected results.

Equations of type \eqref{eq1} with general densities, more precisely
\begin{equation}\label{eq2}
\varrho(x)\partial_{t}u(x,t)=\Delta u^m(x,t), \quad
(x,t)\in\real^{N}\times(0,\infty),
\end{equation}
where $\varrho$ is a density function with suitable behavior, have
been proposed by Kamin and Rosenau in a series of papers \cite{KR81,
KR82, KR83} to model thermal propagation by radiation in
non-homogeneous plasma. Afterwards, a huge development of the
mathematical theory associated to Eq. \eqref{eq2} begun, usually
under conditions such as
$$
\varrho(x)\sim|x|^{-\gamma}, \quad \hbox{as} \ |x|\to\infty,
$$
for some $\gamma>0$, as for example in the following papers
\cite{E90, EK94, RV06, Ted, RV08, RV09, KRV10} where its qualitative
properties and asymptotic behavior are studied. In particular, along
these references, the basic existence and regularity properties are
proved under suitable conditions for the initial data, and a
detailed study for the asymptotic behavior for $\gamma\neq2$ has
been done. Thus, it has been noticed that for $\gamma\in(0,2)$, the
solutions have similar qualitative properties to the ones of the
standard porous medium equation
\begin{equation}\label{PME}
u_t=\Delta u^m,
\end{equation}
see \cite{RV06, RV09}, while for $\gamma>2$ they are quite different
\cite{KRV10}. Thus, the value $\gamma=2$ is critical. A first step
in the study of Eq. \eqref{eq2} with $\varrho(x)\sim|x|^{-2}$ at
infinity, but $\varrho$ regular at $x=0$, has been done in the
recent paper \cite{NR}, having as starting point some conjectures
and comments in \cite{KRV10}.

On the other hand, concerning the asymptotic behavior of solutions
to Eq. \eqref{eq2}, it is shown that the profiles are special
solutions of Eq. \eqref{eq1}, giving thus rise to the natural
problem of the study of the pure power density case
$\varrho(x)=|x|^{-\gamma}$. A special feature of Eq. \eqref{eq1},
besides its general interest for classifying asymptotic profiles for
the general case \eqref{eq2}, is the fact that a strong singularity
appears at $x=0$. As we will see, the presence of this singular
coefficient (in contrast with the above mentioned papers where
$\varrho(x)$ is supposed regular at $x=0$), introduces various
unexpected mathematical phenomena, as the appearance of two regimes
of convergence, different profiles for different initial data only
near $x=0$, and backward evolution of the profiles.

Recently, in a previous work \cite{IS12}, the authors proved some of
these interesting and striking features for the easier case of the
linear equation
\begin{equation}\label{NHheat}
|x|^{-2}\partial_{t}u(x,t)=\Delta u, \quad
(x,t)\in\real^N\times(0,\infty),
\end{equation}
where all the profiles are explicit and one can use the theory of
the heat equation. Moreover, a formal study of the radially
symmetric solutions to Eq. \eqref{eq1} for general $\gamma$ has been
performed in \cite{IRS}, including some mappings that will be useful
in the sequel for the case $\gamma=2$.

Before stating and explaining our main results, we want to mention
that we only consider dimensions $N\geq3$, letting apart the cases
$N=1$ and $N=2$ for a further work, due to some differences with
respect to the general functional theory.

\medskip

\noindent \textbf{Main results.} In the present paper, we deal with
the Cauchy problem associated to Eq. \eqref{eq1} with initial
condition
\begin{equation}\label{initdata}
u_0(x):=u(x,0)\in L^{1}_2(\real^N), \quad u_0\geq0,
\end{equation}
where $N\geq3$ and, as usual,
$$
L^1_{2}(\real^N):=\left\{h:\real^N\mapsto\real, \ h \ {\rm
measurable}, \int_{\real^N}|x|^{-2}h(x)\,dx<\infty \right\}.
$$
In Section \ref{sec.wp} we make a review of the notion of
\emph{strong} solution to \eqref{eq1} and the well-posedness results
we need. In particular, for any initial condition $u_0$ as above,
there exists a strong solution $u$ having some additional
qualitative properties, see Theorem \ref{th.wp}.

As it will become clear from the subsequent analysis, there is a
branching point for the large-time behavior of general solutions to
\eqref{eq1}. Similar to the linear case $m=1$ studied in our
previous paper \cite{IS12}, a big difference is related to whether
$u_0(0)=0$ or $u_0(0)\neq0$. But let us state rigorously our
results.

\medskip

\noindent \textbf{1. Initial data $u_0$ such that $u_0(0)=0$.} Let
us introduce the following weighted norm:
$$
\|h\|_{p,N}=\int_{\real^N}|x|^{-N}|h(x)|^p\,dx
$$
We state first a simple convergence result in the previous integral
norm, which is interesting by itself but can be also seen as a
preliminary.
\begin{theorem}\label{th.1}
Let $u$ be a \emph{radially symmetric} solution to Eq. \eqref{eq1}
with initial condition $u_0$ satisfying\eqref{initdata} and
furthermore
\begin{equation}\label{masscond}
u_0(0)=0, \quad M_{u_0}:=\int_{\real^N}|x|^{-N}u_0(x)\,dx<\infty.
\end{equation}
Then, for any $1\leq p<\infty$, we have
\begin{equation}\label{pconv2}
\lim\limits_{t\to\infty}t^{(p-1)/mp}\|u(t)-F(t)\|_{p,N}=0,
\end{equation}
where
\begin{equation}\label{profile1}
F(x,t)=\left\{\begin{array}{ll}0, & {\rm for} \
|x|<e^{-kt^{1/m}},\\t^{-1/m}\left[-\frac{1}{m(N-2)}\log|x|t^{-1/m}\right]_{+}^{1/(m-1)},
& {\rm for} \ |x|\geq e^{-kt^{1/m}},\end{array}\right.
\end{equation}
and $k=k(M_{u_0})$ has a precise value depending on the mass
$M_{u_0}$ above.
\end{theorem}
We notice that the previous result fails to hold in the $L^{\infty}$
norm, since the limit is discontinuous (is a shockwave). In order to
improve this result and give a more precise description, we use the
convergence in the sense of graphs for multivalued functions. This
is a situation arising in cases of convergence towards discontinuous
solutions of conservation laws, as explained for example in
\cite{EVZ}. Before stating the result, we introduce the necessary
elements, adapted to our case, in the following:
\begin{definition}\label{def.conv}
Let $f,g:D\subseteq\real\mapsto2^{\real}$ be two multivalued
functions. We define the distance between their graphs in the
natural way:
$$
d_{g}(f(x),g(x))=\inf\{|y-z|:y\in f(x), \ z\in g(x)\}, \quad {\rm
for \ any} \ x\in D.
$$
Let $\{f_k\}:D\subseteq\real\mapsto2^{\real}$ be a sequence of
multivalued functions and $F:D\subseteq\real\mapsto2^{\real}$. We
say that $\{f_k\}$ \emph{converges to $F$ in the sense of graphs} if
for any $\e>0$, there exists $k_{\e}$ sufficiently large such that
$$
d_{g}(f_{k}(x),F(x))\leq\e, \quad {\rm for \ any \ } k\geq k_{\e}, \
x\in D.
$$
\end{definition}
We notice that this notion of convergence generalizes the standard
uniform convergence, to which it reduces if all the functions
involved are univalued. In the case of a function $F$ having a jump
discontinuity at some point $x_0\in D$, assuming that
$$
l_{-}=\lim\limits_{x\to x_0^{-}}F(x)<\lim\limits_{x\to
x_0^{+}}F(x)=l_{+},
$$
we will think at it as the multivalued map with
$F(x_0)=[l_{-},l_{+}]$ (and similarly if $l_{-}>l_{+}$) and $F(x)$
as usual at $x\neq x_0$.

In order to apply it in our case, as we only deal with radially
symmetric solutions, we introduce the new variables and functions
\begin{equation}\label{new.coord}
\overline{u}(y,t)=t^{1/m}u(|x|,t), \quad
\overline{F}(y,t)=t^{1/m}F(|x|,t), \quad y=-\log|x|t^{-1/m}.
\end{equation}
Notice that in the new variables, the profile
$\overline{F}(y,t)=[y/m(N-2)]^{1/(m-1)}$ on its support, thus it is
stationary, and presents a unique jump discontinuity at $y=k$. With
the considerations above, we have the following result
\begin{theorem}\label{th.1b}
Let $u_0$, $u$, $F$ as in Theorem \ref{th.1}. Then, in the notation
of \eqref{new.coord}, we have
\begin{equation}\label{conv.graph}
\overline{u}(y,t)\longrightarrow\overline{F}(y), \quad {\rm as} \
t\to\infty,
\end{equation}
the convergence being in the sense of graphs.
\end{theorem}
\noindent Observe that this convergence implies uniform convergence
in the form
$$
\lim\limits_{t\to\infty}t^{1/m}\left|u\left(e^{t^{1/m}\log|x|},t\right)-F\left(e^{t^{1/m}\log|x|},t\right)\right|=0
$$
away from the shock line of $\overline{F}$, and a control of the
maxima of the family $\overline{u}(t)$ at the shock line $y=k$. We
discover thus a striking phenomenon: convergence to a discontinuous
profile presenting a shock line, which usually comes from a
conservation law. But in Eq. \eqref{eq1}, it is not obvious which
internal process of it may give rise to such discontinuous profile.

We represent in Figure \ref{figure1} both the asymptotic profile $F$
and the evolution of a solution, showing the formation of the
shockwave. The numerical experiment has been done for
$u_0(x)=\max\{(x-0.5)(1.5-x),0\}$.

\medskip

\begin{figure}[ht!]
  \begin{center}
  \includegraphics[width=15cm,height=10cm]{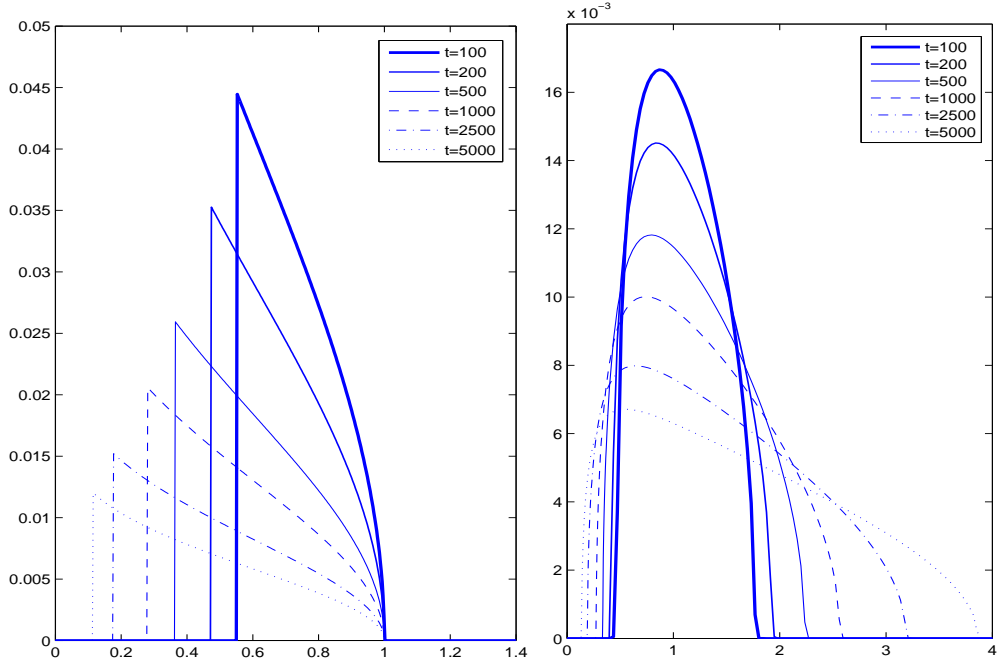}
  \end{center}
  \caption{Profile $F$ (left) and evolution of a general solution (right) for $m=3$ in dimension $N=3$.} \label{figure1}
\end{figure}

\medskip

We remark that in the second figure, some tail appears in the
evolution. The explanation for it is the following: in the
\emph{global time-scale}, this tail for $|x|>1$ is negligible in the
large-time limit (it will tend uniformly to 0). But refining the
scale, one can see some different \emph{localized behavior close to
infinity}, as shown in Theorem \ref{th.3} below.

For general (not necessarily radially symmetric) solutions, we can
get the optimal time decay rate applying the comparison principle.
\begin{corollary}\label{cor.decay}
Let $u$ be a \emph{general solution} with initial condition $u_0\in
L^{\infty}(\real^N)$ and as in Theorem \ref{th.1}, such that there
exists $r>0$ with
\begin{equation}\label{inf0}
\min\limits_{|x|=r}u_0(x)>0.
\end{equation}
Then, the optimal time decay rate of $\|u(t)\|_{\infty}$ is
$t^{-1/m}$.
\end{corollary}

\medskip

\noindent \textbf{2. Initial data $u_0$ such that $u_0(0)=K>0$.} In
this case, we can give a general result, showing that the value $K$
at the origin is preserved along the evolution, and that we have a
uniform convergence in the standard sense towards a continuous
profile. We state first the result for radially symmetric functions,
as follows
\begin{theorem}\label{th.2}
Let $u$ be a \emph{radially symmetric} solution to Eq. \eqref{eq1}
with initial condition $u_0$ satisfying \eqref{initdata},
$u_0(0)=K>0$ and furthermore
\begin{equation}\label{initdata2}
u_0 \ {\rm continuous \ at \ x=0}, \quad 0\leq u_0(x)\leq K, \ {\rm
for \ any} \ x\in\real^N, \quad \lim\limits_{|x|\to\infty}u_0(x)=0,
\end{equation}
Then, we have
\begin{equation}\label{conv.unif}
\lim\limits_{t\to\infty}|u(x,t)-E_{K}(x,t)|=0,
\end{equation}
uniformly in $\real^N$, where
\begin{equation}\label{profile2}
E_{K}(x,t)=\left\{\begin{array}{ll}K, & {\rm if} \ 0\leq|x|\leq
e^{-mK^{m-1}(N-2)t},
\\ \left[-\frac{1}{m(N-2)}\frac{\log|x|}{t}\right]^{1/(m-1)}, & {\rm
if} \ e^{-mK^{m-1}(N-2)t}<|x|<1, \\0, & {\rm if} \
|x|\geq1,\end{array}\right.
\end{equation}
In particular, we also have that $u(0,t)=K$, for any $t>0$.
\end{theorem}
The limit profiles obtained in Theorems \ref{th.1} and \ref{th.2}
are interesting and difficult to guess at first sight. Indeed, they
are apparently not related to the equation \eqref{eq1} itself, but
they are obtained via an \emph{asymptotic simplification} process
applied to some equation of porous medium type with convection, in
which \eqref{eq1} can be mapped via a transformation indicated in
Section \ref{sec.tr}.

In the general case of non-radially symmetric solutions, we need to
impose an estimate of the decay at infinity. More precisely, we
have:
\begin{theorem}\label{th.2b}
Let $u$ be a \emph{general} solution to Eq. \eqref{eq1}, with
initial condition $u_0$ satisfying \eqref{initdata}, $u_0(0)=K>0$,
\eqref{initdata2} and furthermore, that there exist $\delta>0$ small
and $R>0$ large such that
\begin{equation}\label{initcond3}
u_0(x)\leq|x|^{2-N-\delta}, \quad {\rm for \ |x|>R}.
\end{equation}
Then, \eqref{conv.unif} holds true uniformly in $\real^N$, with the
same profile $E_K$ as in Theorem \ref{th.2}.
\end{theorem}
\begin{remark}\label{rmk}
Condition \eqref{initcond3} in Theorem \ref{th.2b} can be made
slightly more general in the following form: there exists a function
$\Psi:\real\to[0,\infty)$ and some $R>0$ such that
\begin{equation*}
u_0(x)\leq\Psi(|x|), \quad {\rm for \ any \ } x\in\real^N, \ |x|>R,
\quad {\rm and} \ \int_{R}^{\infty}r^{N-3}\Psi(r)\,dr<\infty.
\end{equation*}
\end{remark}
We represent in Figure \ref{figure2} both the profile $E_K$ and the
evolution of a general solution, showing how its form approaches the
expected one. There is again a problem with the tails for $|x|>1$,
as the outer time-scale is different from the global one. This will
be explained below in Theorem \ref{th.3}.

\medskip

\begin{figure}[ht!]
  \begin{center}
  \includegraphics[width=15cm,height=10cm]{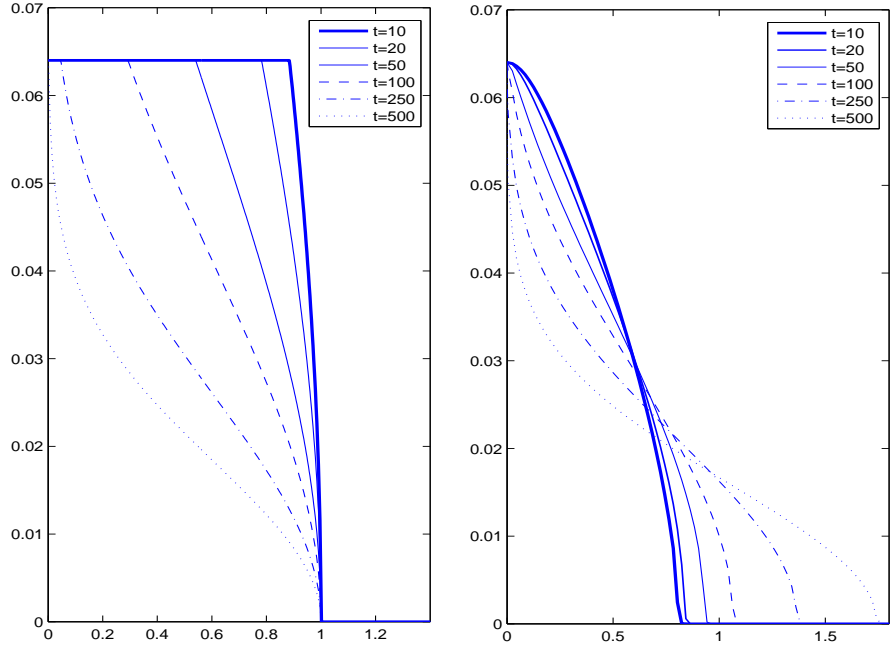}
  \end{center}
  \caption{Profile $E_K$ and evolution of a general solution for $m=3$ in dimension $N=3$.} \label{figure2}
\end{figure}

\medskip

Let us notice some curious facts resulting from our analysis.
\begin{remark}[Lack of continuity as $m\to1$]
We have discovered a striking phenomenon: there is a lack of
"continuity" as $m\to1$ in our equation. Indeed, when considering
the standard porous medium equation \eqref{PME} and the heat
equation, one notices that there appears some kind of "continuity"
as $m\to1$, at some formal level; just as an example, the optimal
decay rates for the porous medium equation and the heat equation in
$\real^N$ are $t^{-N/(mN-N+2)}$, respectively $t^{-N/2}$, and they
are obviously continuous as $m\to1$; the same happens about the
similarity exponents of the fundamental solutions of both equations.

Recalling the analysis done for the case $m=1$, that is
$$
|x|^{-2}u_t=\Delta u,
$$
in our previous paper \cite{IS12}, we notice that the above
considerations are not true for Eq. \eqref{eq1} with $m\geq1$: the
optimal decay rates (when dealing with initial data $u_0(0)=0$) are
$t^{-1/m}$ for $m>1$, respectively $t^{-1/2}$ for $m=1$ (see
\cite{IS12}).

This is very surprizing, and the explanation is the following: Eq.
\eqref{eq1} can be mapped, for all $m\geq1$, by the transformation
\eqref{tr1}, into the general convection-diffusion equation
$$
w_{\tau}=(w^m)_{ss}-(w^m)_s, \quad {\rm in} \ \real\times(0,\infty).
$$
But the properties of it for $m>1$ depart strongly from the ones of
the linear case $m=1$, as shown in \cite{EVZ, LS98}. This explains
our apparently strange "lack of continuity" in the exponents.
\end{remark}

\medskip

\noindent \textbf{3. Finer asymptotics in the "outer" region.} We
can improve the previous convergence results in regions that are far
away from the origin by refining the time-scale (making a "zoom" on
that region), as also the numerical experiments above had shown by
the appearance of thin tails for $|x|>1$ at any $t>0$. We recall
that there exists a family of explicit solutions to Eq. \eqref{eq1}:
\begin{equation}\label{Bar}
B_D(x,t)=t^{-1/(m-1)}\left[D-\frac{1}{m(N-2)}\log(|x|t^{-1/(m-1)(N-2)})\right]_{+}^{1/(m-1)},
\quad {\rm for \ any \ } D>0,
\end{equation}
which does not enter into the framework of our previous study, since
it is singular at $x=0$. But it gives a better information with
respect to large time behavior and decay rates in regions that are
far from the origin; more precisely, we have
\begin{theorem}\label{th.3}
Let $u$ be a solution to Eq. \eqref{eq1} with initial datum $u_0$
satisfying \eqref{initdata}. Then, for any $\delta>0$, we have
$$
\lim\limits_{t\to\infty}t^{1/(m-1)}|u(t)-B_{D}(t)|=0, \quad {\rm
uniformly \ in \ the \ set } \ \left\{|x|\geq\delta
t^{1/(m-1)(N-2)}\right\},
$$
where $D$ is the (unique) constant in \eqref{Bar} such that
$\|B_D\|_{L^1_2}=\|u_0\|_{L^1_2}$.
\end{theorem}
The proof is a straightforward adaptation of the one of
\cite[Theorem 1.1]{NR}, as all the technical steps there are done in
compact sets of $\real^N\setminus\{0\}$ where the singularity at
$x=0$ does not have any effects. We omit the details.

Notice that Theorem \ref{th.3} gives an optimal information in
\emph{outer sets} of the type $\{|x|\geq Ct^{1/(m-1)(N-2)}\}$, where
our analysis was only giving a non-optimal decay rate due to the
domination (in the global profiles) of the \emph{inner regions}
close to the origin. A similar situation was described in
\cite[Section 7]{KRV10}.

\medskip

\noindent \textbf{Organization of the paper.} In Section
\ref{sec.wp} we give the notion of a solution to \eqref{eq1} and we
recall some well-posedness results proved in \cite{NR}. Section
\ref{sec.tr} deals with a transformation mapping the radially
symmetric solutions to \eqref{eq1} to solutions to a porous medium
equation with convection that is also interesting in itself. Based
on this transformation, we prove Theorems \ref{th.1} and \ref{th.1b}
in Section \ref{sec.th1} using previous knowledge on the transformed
equation. The proof of \ref{th.2} is done in Section \ref{sec.th2}
and relies on the four-step technique, which is a general method in
proving large time behavior for nonlinear diffusion equations, see
\cite{V, Re97, NR} for other examples of use; the asymptotic
simplification will arise then in a natural way. Then, passing to
general solutions and proving Theorem \ref{th.2b} and Corollary
\ref{cor.decay} is an application of the comparison principle, see
Section \ref{sec.general}. We finally include a Section
\ref{sec.final} on further extensions to other equations or
densities of interest and some open problems.

\section{Well-posedness. Functional preliminaries}\label{sec.wp}

In this section we make precise our notions of solution that we use
along the present paper. We deal with the Cauchy problem
\begin{equation}\label{CP1}
\left\{\begin{array}{ll}|x|^{-2}u_t=\Delta u^m, \quad \hbox{in} \
\real^N\times(0,\infty),\\ u(x,0)=u_0(x), \quad
x\in\real^N,\end{array}\right.
\end{equation}
with $u_0$ satisfying \eqref{initdata}. The theory of existence and
uniqueness for this Cauchy problem has been recently studied in
\cite{NR}, from where we extract some statements for the sake of
completeness. Let
$$
Q^*:=\real^N\times(0,\infty)\setminus\{(0,t):t>0\}.
$$
We do not insist on the notion of \emph{weak solution} to
\eqref{eq1}, which is a straightforward adaptation of
\cite[Definition 2.1]{NR} to our special density
$\varrho(x)=|x|^{-2}$. We next consider the class $\cf$ of functions
$u\in C([0,\infty);L^1_2)\cap C(Q^*)$ satisfying the following
limitations on their behavior near $x=0$ and close to infinity: for
every $\tau>0$, there exists a constant $C(\tau)$ such that
$$
u(x,t)\leq C(\tau)|\log x|^{1/(m-1)}, \quad {\rm for} \ {\rm any} \
t>\tau, \ 0<|x|<\e<1
$$
for some $\e>0$ sufficiently small, and
$$
u(x,t)\leq C(\tau), \quad {\rm for} \ {\rm any} \ t>\tau, \ |x|>1.
$$
With these conditions, we have the following
\begin{theorem}\label{th.wp}
Let $u_0$ as in \eqref{initdata}. Then the Cauchy problem
\eqref{CP1} admits a unique (weak) solution $u\in\cf$ such that
\begin{equation}\label{strongsol}
|x|^{-2}u_t, \ \Delta u^m\in L^1_{\rm loc}(Q_*), \quad
|x|^{-2}u_t=\Delta u^m \ a.\,e. \ {\rm in} \ Q_*.
\end{equation}
\end{theorem}
This theorem is proved in \cite[Theorem 3.1]{NR}. Following previous
literature, a solution $u\in\cf$ satisfying \eqref{strongsol} will
be called a \emph{strong solution} in the sequel. Moreover, if $u_0$
is continuous, then the solution $u$ is continuous including at the
origin.

We end these preliminaries with the following contraction principle,
also proved in \cite{NR}, that implies both uniqueness and
comparison:
\begin{proposition}[$L^{1}_2$-Contraction principle]\label{prop.contraction}
Let $u_1$, $u_2$ be two strong solutions of Eq. \eqref{eq1}. For
$0<t_1<t_2$ we have
\begin{equation}\label{contr}
\int_{\real^N}|x|^{-2}\left[u_1(x,t_2)-u_2(x,t_2)\right]_{+}\,dx\leq\int_{\real^N}|x|^{-2}\left[u_1(x,t_1)-u_2(x,t_1)\right]_{+}\,dx,
\end{equation}
where $[g]_{+}$ represents the positive part of a generic function
$g$.
\end{proposition}

\section{Radially symmetric solutions. The
transformation}\label{sec.tr}

We now restrict ourselves to radially symmetric solutions
$u(x,t)=u(r,t)$, $r=|x|$, to \eqref{eq1}. We introduce the following
change of variables:
\begin{equation}\label{tr1}
u(r,t)=w(s,\tau), \quad r=e^{\theta s}, \ t=\theta^2\tau,
\end{equation}
for some $\theta$ to be chosen later. Denoting by subindex the
derivative with respect to the corresponding variable, we notice
that
$$
(u^m)_r(r,t)=\frac{1}{\theta r}(w^m)_s(s,\tau), \quad
(u^m)_{rr}(r,t)=\frac{1}{(\theta
r)^2}(w^m)_{ss}(s,\tau)-\frac{1}{\theta r^2}(w^m)_s(s,\tau).
$$
Replacing these formulas into the radially symmetric form of Eq.
\eqref{eq1}, that is
$$
r^{-2}\partial_{t}u(r,t)=(u^m)_{rr}(r,t)+\frac{N-1}{r}(u^m)_r(r,t),
$$
we arrive to the following equation satisfied by $w$:
\begin{equation}\label{transfeq}
\partial_{\tau}w(s,\tau)=(w^m)_{ss}(s,\tau)+(N-2)\theta(w^m)_{s}(s,\tau).
\end{equation}
As we are in dimension $N\geq3$, we arrive to a porous medium
equation with convection in one space dimension. We can simply
choose then $\theta=-1/(N-2)<0$ to get
\begin{equation}\label{PMEconv}
\partial_{\tau}w(s,\tau)=(w^m)_{ss}(s,\tau)-(w^m)_{s}(s,\tau).
\end{equation}

The equation \eqref{PMEconv} is very interesting by itself and it
has been obtained from the modelling of the transport of a solute
through a porous medium, the physical models appearing in \cite{GDD}
and references therein. Starting from this point, the mathematical
theory for \eqref{PMEconv} (and also for the one with a plus sign in
the right hand side, which is equivalent to \eqref{PMEconv} changing
$s$ into $-s$) developed rapidly in the framework of the so-called
\emph{mild solutions}, generated via semigroup theory, as for
example in \cite{BT, DK, G1} and references therein. In particular,
well-posedness for the Cauchy problem is proved when $w_0\in
L^1(\real)$. The large-time behavior of solutions to \eqref{PMEconv}
has been studied by Lauren\ced{c}ot and Simondon in \cite{LS97,
LS98}, where again they ask for the initial data $w_0$ to be
integrable.

\medskip

\noindent \textbf{Remark.} The same transformation applies in
dimension $N=2$ leading to the standard porous medium equation
\eqref{PME}. We do not deal with this case here, as explained in the
Introduction.

\section{Radially symmetric solutions. Asymptotic behavior when
$u_0(0)=0$}\label{sec.th1}

In this section we prove Theorems \ref{th.1} and \ref{th.1b}. We
will work with the convection-diffusion equation \eqref{PMEconv}
obtained via the transformation \eqref{tr1}. A specific fact in this
type of equations is the competition between the two processes
(nonlinear diffusion and convection), leading to different ranges
where the large time behavior is very different. In our case, it is
shown in \cite{LS98} that we are in the range where the convection
process is dominant.

Before beginning the rigorous proof, we have to "guess" the correct
profile. This shows how the asymptotic simplification comes into
play, once we make the \emph{ansatz}:
\begin{equation}\label{ansatz1}
w(s,\tau)=\tau^{-1/m}v(y,\tau), \quad y=s\tau^{-1/m}.
\end{equation}
Notice that the ansatz is coherent with our case, of initial
condition vanishing at the origin (that is, after the transformation
\eqref{tr1}, that $\lim\limits_{s\to\infty}w(s,0)=0$), as we expect
to have a time decay. Thus, we calculate
$$
w_{\tau}(s,\tau)=\tau^{-1/m}v_{\tau}(y,\tau)-\frac{1}{m}\tau^{-1-1/m}\left(v(y,\tau)+yv_{y}(y,\tau)\right)
$$
and
$$
(w^m)_s(s,\tau)=\tau^{-1-1/m}(v^m)_{y}(y,\tau), \quad
(w^m)_{ss}(s,\tau)=\tau^{-1-2/m}(v^m)_{yy}(y,\tau),
$$
hence, letting also $\overline{\tau}=\log(1+\tau)$, we deduce that
$v=v(y,\overline{\tau})$ solves
\begin{equation}\label{transfeq2}
v_{\overline{\tau}}=\frac{1}{m}(v+yv_{y})-(v^m)_y+e^{-\overline{\tau}/m}(v^m)_{yy}.
\end{equation}
Letting formally $\overline{\tau}\to\infty$ in \eqref{transfeq2} and
assuming that the time decay in \eqref{ansatz1} is the correct one,
we expect to converge to a stationary solution $v^*=v^*(y)$ solving
$$
\frac{1}{m}\left(yv^*_{y}+v^*\right)-\left[(v^*)^m\right]_y=0.
$$
By integration and taking into account that
$\lim\limits_{y\to\infty}v^*(y)=0$, we find
$v^*(y)=(y/m)_{+}^{1/(m-1)}$, hence, coming back to the original
variables, we expect that the limit profile of $w$ as
$\tau\to\infty$ will be
\begin{equation}\label{lim.conv}
W(s,\tau)=\left\{\begin{array}{ll}\tau^{-1/m}\left(\frac{1}{m}s\tau^{-1/m}\right)^{1/(m-1)},
& {\rm for} \ s\in[0,k\tau^{1/m}),\\0, & {\rm
otherwise},\end{array}\right.
\end{equation}
which coincides with the one in \cite[Theorem 1.4]{LS98} and is a
special solution to the following first order equation obtained via
asymptotic simplification
$$
W_{\tau}+(W^m)_{s}=0.
$$
The constant $k$ for the branching point above is unique and is
chosen for such profile to have initial mass $M_{u_0}$. Finally, in
the initial equation, the limit profile writes
$$
F(x,t)=\left\{\begin{array}{ll}0, & {\rm for} \
|x|<e^{-kt^{1/m}},\\t^{-1/m}\left[-\frac{1}{m(N-2)}\log|x|t^{-1/m}\right]_{+}^{1/(m-1)},
& {\rm for} \ |x|\geq e^{-kt^{1/m}},\end{array}\right.
$$
as expected in Theorem \ref{th.1}.

All the previous calculations were totally formal; they show how the
limit profile appears in a logical manner in our work. We are now in
position to provide the rigorous proofs.

\medskip

\begin{proof}[Proof of Theorem \ref{th.1}]
This is simple in view of the results in \cite{LS98}. Let $u_0$ be
an initial condition as in the statement of Theorem \ref{th.1}, $u$
the (radially symmetric) solution to Eq. \eqref{eq1} with initial
condition $u_0$ and $w$ be the solution to Eq. \eqref{PMEconv}
obtained from $u$ via the tranformation \eqref{tr1}. Then, condition
\eqref{masscond} implies that $\lim\limits_{s\to\infty}w_0(s)=0$ and
that
$$
\int_{-\infty}^{\infty}w_0(s)\,ds=(N-2)\int_0^{\infty}\frac{u_0(r)}{r}=\frac{(N-2)M_{u_0}}{\omega_1}<\infty,
$$
where $\omega_1$ is the area of the unit sphere in $\real^N$. It
follows that $w_0\in L^1(\real)$. We are in the same conditions as
in \cite[Theorem 1.4]{LS98} (for the case $q=m$ in the notations
used there), hence we deduce that, for any $p\in[1,\infty)$, we have
\begin{equation}\label{L1conv}
\lim\limits_{t\to\infty}t^{(p-1)/mp}\|w(t)-W(t)\|_{p}=0,
\end{equation}
which, undoing the change of variables, yields \eqref{pconv2}.
\end{proof}

\medskip

\begin{proof}[Proof of Theorem \ref{th.1b}]
In order to prove the convergence in the sense of graphs, we use the
following result, which is an adaptation of \cite[Lemma 2.2]{EVZ}.
\begin{lemma}\label{lem.graph}
Let $g\in L^1(\real)$ be a nonnegative function such that
\begin{equation}\label{interm1}
(g^{m-1})_x\leq1, \quad {\rm for \ any \ } x>0.
\end{equation}
Consider the function $G$ defined as
\begin{equation*}
G(x)=\left\{\begin{array}{ll}x^{1/(m-1)}, & {\rm for} \ 0\leq x\leq T\\
0, & {\rm otherwise},\end{array}\right.
\end{equation*}
where $T$ is chosen such that $\|G\|_{1}=\|g\|_1$. Assume that
$\|g-G\|_{1}<\e$ for some $\e>0$. Then the distance between the
graphs of $g$ and $G$ can be estimated by a positive power of $\e$.
\end{lemma}
Notice that in \cite[Lemma 2.2]{EVZ} there is a weaker integral
condition, which is implied by our condition on the distance in
$L^1$, and the lemma is more general, holding true for any $k>0$
instead of $m-1$.

We want to apply it for $g=w$ solution to Eq. \eqref{PMEconv} and
$G=W$, the limit profile in \eqref{lim.conv}. In order to do it, we
need to prove that its conditions are fulfilled. The following
result, which is interesting by itself, will imply \eqref{interm1}.

\begin{lemma}\label{lem.est}
Let $w$ be a solution to \eqref{PMEconv} with initial condition
$w_0\in L^1(\real)\cap L^{\infty}(\real)$. Then
\begin{equation}\label{entropy}
\left(w^{m-1}\right)_{s}\leq\frac{1}{m\tau}
\end{equation}
\end{lemma}
\begin{proof}
We use a Bernstein technique. At a formal level, consider
$$
z(s,\tau):=\frac{m}{m-1}w(s,\tau)^{m-1}.
$$
Then by straightforward calculations (see \cite{Re97}), the equation
satisfied by $z$ is
\begin{equation}\label{interm2}
z_{\tau}=(m-1)zz_{ss}+z_{s}^2-(m-1)zz_s.
\end{equation}
Let then $p=z_s$. By differentiating in \eqref{interm2}, we obtain
the equation solved by $p$:
\begin{equation}\label{interm3}
p_{\tau}=(m+1)pp_s+(m-1)zp_{ss}-(m-1)p^2-(m-1)zp_s.
\end{equation}
Let then $\overline{p}(\tau)=1/(m-1)\tau$. We notice that
$\overline{p}$ is a solution to \eqref{interm3} and
$\overline{p}(s,0)=+\infty>p(s,0)$ for any $s\in\real$. Thus,
$\overline{p}$ is a supersolution to our problem, whence by standard
comparison we get \eqref{entropy}.

The above is a formal calculation, that holds true rigorously for
solutions that are uniformly positive. Thus, for a rigorous proof,
we have to consider solutions having $0<\e\leq w_0(s)$, for which
all previous calculations apply, then approximate as $\e\to0$. We
omit the details as this last technical step is quite standard (see
\cite[Lemma 1.1]{EVZ}, \cite[Lemma 2.10]{LS98}, in the latter a
fully detailed proof of such approximation being given).
\end{proof}
We are now in position to check the conditions in Lemma
\ref{lem.graph}. Consider the new function and variable
$$
\overline{w}(y,\tau)=\tau^{1/m}w(s\tau^{-1/m},\tau), \quad
y=s\tau^{-1/m}
$$
and notice that, for any $\tau>0$, we have
\begin{equation}\label{interm7}
\frac{\partial}{\partial
y}\overline{w}^{m-1}(y,\tau)=\tau\frac{\partial}{\partial
s}w^m(s\tau^{-1/m},\tau)\leq\frac{1}{m}.
\end{equation}
In the new variables $(y,\tau)$, we apply Lemma \ref{lem.graph} for
the following functions:
$$
g(y,\tau)=m^{1/(m-1)}\overline{w}(y,\tau), \quad
G(y,\tau)=m^{1/(m-1)}\tau^{1/m}W(y,\tau),
$$
where $W$ is the profile in \eqref{lim.conv}. Notice that, in the
new variables, $G$ has the required form, and the convergence in
$L^1$ is insured by \eqref{L1conv}. The bound for $(g^{m-1})_{y}$
follows from \eqref{interm7}. Thus, an application of Lemma
\ref{lem.graph} gives that $\overline{w}$ converges to $W$ in the
sense of graphs in the new variables. We end the proof by undoing
the change of variables $(s,\tau)\mapsto(y,\tau)$ and transformation
\eqref{tr1}.
\end{proof}
\begin{remark}
Let us notice that, in the above proof, we show in particular that
solutions $w$ to \eqref{PMEconv} satisfy
$$
\tau^{1/m}w(y,\tau)\longrightarrow\tau^{1/m}W(y,\tau), \quad
y=s\tau^{-1/m},
$$
with convergence in the sense of graphs. This is a slight
improvement of \cite[Theorem 1.4]{LS98}.
\end{remark}

\section{Radially symmetric solutions. Asymptotic behavior when
$u_0(0)=K>0$}\label{sec.th2}

In this section we prove Theorem \ref{th.2}. Similar to the previous
section, we begin with a formal calculation which will give us a
guess of the profile. As we expect the value $K>0$ at the origin to
maintain, no time decay is allowed in this case, thus we start from
another ansatz to plug in \eqref{PMEconv}:
\begin{equation}\label{ansatz2}
w(s,\tau)=v\left(\frac{s}{\tau},\tau\right), \quad y=\frac{s}{\tau},
\quad \overline{\tau}=\log(1+\tau).
\end{equation}
By straightforward calculations, we obtain that
$v=v(y,\overline{\tau})$ satisfies the following equation
\begin{equation}\label{interm4}
v_{\overline{\tau}}-yv_{y}=(v^m)_{y}+\frac{1}{e^{\overline{\tau}}-1}(v^m)_{yy}.
\end{equation}
We are again in a case of asymptotic simplification where the effect
of the diffusion term is negligible in the limit. Passing formally
to the limit as $\overline{\tau}\to\infty$ in \eqref{interm4}, and
assuming the limit $v^*=v^*(y)$ to be stationary in the new
variables, we deduce that $v^*$ solves the following equation
$$
yv^*_y=mv^{m-1}v^*_y,
$$
hence, either $v^*$ is constant (in some interval), or
$v^*(y)=(y/m)^{1/(m-1)}_{+}$ in the complementary part. As the
constant part is expected to be equal to the initial value $K$, and
undoing the change of variables \eqref{ansatz2}, we expect the
asymptotic profile for \eqref{PMEconv} in this case to be given by
\begin{equation}\label{lim.conv2}
V(s,\tau)=\left\{\begin{array}{lll}0, & {\rm if} \ s\leq0,
\\ \left[\frac{s}{m\tau}\right]^{1/(m-1)}, & {\rm if} \
0<s<mK^{m-1}\tau, \\ K, & {\rm if} \ s\geq
mK^{m-1}\tau,\end{array}\right.
\end{equation}
Notice that this function is continuous, departing strongly from the
profile $W$ introduced in \eqref{lim.conv}, which develops a shock
curve. Undoing now the transformation \eqref{tr1}, we arrive to our
expected profile
$$
E_{K}(x,t)=\left\{\begin{array}{llll}K, & {\rm if} \ 0\leq|x|\leq
e^{-mK^{m-1}(N-2)t},
\\ \left[-\frac{1}{m(N-2)}\frac{\log|x|}{t}\right]^{1/(m-1)}, & {\rm
if} \ e^{-mK^{m-1}(N-2)t}<|x|<1, \\0, & {\rm if} \
|x|\geq1,\end{array}\right.
$$
which coincides with the one in Theorem \ref{th.2}.

All these calculations are, obviously, formal, showing how the
profile $E_{K}$ arises. We are now able to prove rigorously the
large-time convergence towards the profile $E_{K}$.

\medskip

\begin{proof}[Proof of Theorem \ref{th.2}]
Let $u_0$, $u$ as in Theorem \ref{th.2}. By the transformation
\eqref{tr1}, we obtain a solution $w$ to \eqref{PMEconv}, such that
$0\leq w_0(s)\leq K$ for any $s\in\real$, and
$\lim\limits_{s\to\infty}w_0(s)=K$,
$\lim\limits_{s\to-\infty}w_0(s)=0$. The proof is divided into two
big steps: first, we reduce the problem to the case when $w_0$ is
nondecreasing (or equivalently the initial data $u_0$ in initial
variables is nonincreasing), and second, we prove the theorem under
this extra hypothesis.

\medskip

\noindent \textbf{Big step A: Reduction to the case of nondecreasing
initial data.}

\medskip

This is based on the following standard result.
\begin{lemma}\label{lem.decr}
Let $w$ be a solution to \eqref{PMEconv} as above such that its
initial condition $w_0$ is nondecreasing. Then $w(\tau)$ is
nondecreasing in $s$ for any $\tau>0$.
\end{lemma}
We only sketch the proof, see also \cite{ILV}[Lemma 3.3].
\begin{proof}
The general principle is to derive the equation satisfied by the
derivative $w_s$. As we work with nonnegative functions, it suffices
to derive it for any power, in particular for
$p=m/(m-1)(w^{m-1})_s$, which is \eqref{interm3}. As \eqref{interm3}
is parabolic, fulfills a comparison principle and $p\equiv0$ is a
solution, it follows that $(w^{m-1}(\cdot,\tau))_s\geq0$ for any
$\tau>0$, whence $w^{m-1}$ is nondecreasing, hence also $w$.
\end{proof}
Suppose that Theorem \ref{th.2} is proved for $w_0$ nondecreasing
(which is equivalent in the initial variables to $u_0$
nonincreasing). Let now a general solution $u$ to \eqref{eq1} such
that $u_0$ satisfies \eqref{initdata2}. Pass again to $w$ solution
to \eqref{PMEconv}, where $0\leq w_0(s)\leq K$ for any $s\in\real$.
Since
\begin{equation}\label{interm9}
\lim\limits_{s\to-\infty}w_0(s)=0, \quad
\lim\limits_{s\to\infty}w_0(s)=K,
\end{equation}
we can easily find some $w_0^{l}$, $w_0^{u}$ which are
nondecreasing, satisfying the same limits as in \eqref{interm9} and
such that
$$
w_0^l(s)\leq w_0(s)\leq w_0^{u}(s), \quad {\rm for \ any \ }
s\in\real.
$$
Let $w^l$, $w^u$ be the solutions to \eqref{PMEconv} with initial
data $w_0^l$, $w_0^u$ respectively. Then, by standard comparison and
Lemma \ref{lem.decr}, $w^l(\tau)$, $w^u(\tau)$ are nondecreasing
with respect to $s$ at any time $\tau>0$, and $w^l(s,\tau)\leq
w(s,\tau)\leq w^u(s,\tau)$, for any
$(s,\tau)\in\real\times(0,\infty)$. Applying Theorem \ref{th.2}
(supposed to be already known for nondecreasing solutions) for
$w^l$, $w^u$, we get the desired convergence result for $w$, whence
for our solution $u$ after undoing the transformation \eqref{tr1}.

\medskip

\noindent \textbf{Big step B: Proof for $w_0$ nondecreasing.}

\medskip

From now on, in all this section we work with solutions $w$ to
\eqref{PMEconv} with nondecreasing initial data $w_0$ as above. We
employ the four-step method, which is by now a standard general
strategy of proving large time behavior for nonlinear diffusion
equations (see e. g. \cite{NR, Re97, V}). We have to adapt the
technique to the less usual case when the maximum order term will
have no effect for large times, as it happens for Eq.
\eqref{PMEconv}.

\medskip

\noindent \textbf{Step 1. Rescaling.} Define for any $\lambda>0$,
$$
w_{\lambda}(s,\tau)=w(\lambda s,\lambda\tau).
$$
Then, $w_{\lambda}$ solves the following equation:
\begin{equation}\label{resc.eq}
(w_{\lambda})_{\tau}=\frac{1}{\lambda}(w_{\lambda}^m)_{ss}-(w_{\lambda}^m)_s.
\end{equation}
This already suggests the asymptotic simplification we expect to
get.

\medskip

\noindent \textbf{Step 2. Uniform estimates.} We want to obtain
estimates for $w_{\lambda}$ that do not depend on $\lambda$. Since
$w_0\in L^{\infty}(\real)$, we readily get (by standard comparison)
that $|w(s,\tau)|\leq\|w_0\|_{\infty}$, whence
\begin{equation}\label{unif.est1}
|w_{\lambda}(s,\tau)|\leq\|w_0\|_{\infty}, \quad {\rm for \ any} \
s\in\real, \ \tau>0, \ \lambda>0.
\end{equation}
Moreover, it is proved in \cite[Lemma 2.10]{LS98} that in our
conditions,
$$
\left|(w^m)_s(s,\tau)\right|\leq\frac{2}{m-1}\|w_0\|_{\infty}t^{-1},
$$
where we can add the modulus in the left-hand side since
$(w^m)_s\geq0$, as $w(\cdot,\tau)$ is nondecreasing in $s$ for any
$\tau>0$, due to Lemma \ref{lem.decr}. Hence
\begin{equation}\label{unif.est2}
\left|(w_{\lambda}^m)_s(s,\tau)\right|=\lambda\left|(w^m)_s(\lambda
s,\lambda\tau)\right|\leq\frac{2\lambda}{m-1}\|w_0\|_{\infty}(\lambda
t)^{-1}=\frac{2}{m-1}\|w_0\|_{\infty}t^{-1}.
\end{equation}
Both estimates \eqref{unif.est1} and \eqref{unif.est2} are uniform
with respect to $\lambda$.

\medskip

\noindent \textbf{Step 3. Passage to the limit.} Due to the previous
estimates, we obtain that the family $\{w_{\lambda}\}$ is uniformly
equicontinuous in compact subsets. By Ascoli-Arzelá Theorem, there
exists a subsequence (not relabeled) $\{w_{\lambda}\}$ that
converges uniformly on compact sets to some limit profile
$w_{\infty}$. We can then pass to the limit in the weak formulation
of the equation \eqref{resc.eq}. Recall that $w_{\lambda}$ satisfies
that
$$
\int_Q\left[\Phi_{s}\left(\frac{1}{\lambda}(w_{\lambda}^m)_s-w_{\lambda}^m\right)-\Phi_{\tau}w_{\lambda}\right]\,ds\,d\tau=0,
\quad {\rm for \ any \ } \Phi\in\cd(Q), \ Q=\real\times[0,\infty).
$$
Since $w_{\lambda}\to w_{\infty}$ uniformly on compact sets (in
particular on the support of $\Phi$) and $(w_{\lambda}^m)_s$ is
bounded uniformly with respect to $\lambda$, we let
$\lambda\to\infty$ to get that
$$
\int_Q\left[\Phi_{s}w_{\infty}^m+\Phi_{\tau}w_{\infty}\right]\,ds\,d\tau=0,
\quad {\rm for \ any \ } \Phi\in\cd(Q),
$$
whence $w_{\infty}$ is a weak solution to the conservation law
\begin{equation}\label{CL}
w_{\infty,\tau}+(w_{\infty}^m)_s=0.
\end{equation}

\medskip

\noindent \textbf{Step 4. Identification of the limit.} It remains
to show that $w_{\infty}=V$, where $V$ is given in
\eqref{lim.conv2}. To this end, we show first that $w_{\infty}$
takes a Heaviside function as initial trace, that is
\begin{equation*}
\lim\limits_{\tau\to0}w_{\infty}(s,\tau)=\left\{\begin{array}{ll}K,
& {\rm if} \ s\geq0, \\ 0, & {\rm if} \ s<0,\end{array}\right.
\end{equation*}
in the sense of distributions, which is equivalent to prove that
\begin{equation}\label{interm5}
\lim\limits_{\tau\to0}\left[\int_{-\infty}^{\infty}w_{\infty}(s,\tau)\Phi(s)\,ds-K\int_{0}^{\infty}\Phi(s)\,ds\right]=0,
\end{equation}
for any $\Phi\in\cd(\real)$. For any $\Phi\in\cd(\real)$, we
estimate:
\begin{equation*}
\begin{split}
\left|\int_{-\infty}^{\infty}\right.&\left.(w_{\lambda}(s,\tau)-w_{\lambda}(s,0))\Phi(s)\,ds\right|=\left|\int_{-\infty}^{\infty}\int_0^{\tau}w_{\lambda,\tau}(s,\theta)\Phi(s)\,d\theta\,ds\right|\\
&=\left|\int_0^{\tau}\int_{-\infty}^{\infty}\left[\frac{1}{\lambda}(w_{\lambda}^m)_{ss}(s,\theta)-(w_{\lambda}^m)_{s}(s,\theta)\right]\Phi(s)\,ds\,d\theta\right|\\
&=\left|\int_0^{\tau}\left[\frac{1}{\lambda}\int_{-\infty}^{\infty}(w_{\lambda}^m)(s,\theta)\Phi_{ss}(s)\,ds+\int_{-\infty}^{\infty}(w_{\lambda}^m)(s,\theta)\Phi_{s}(s)\,ds\right]\,d\theta\right|\\
&\leq\int_0^{\tau}\left|\int_{-\infty}^{\infty}(w_{\lambda}^m)(s,\theta)\Phi_s(s)\,ds\right|\,d\theta+\frac{1}{\lambda}\int_0^{\tau}\left|\int_{-\infty}^{\infty}(w_{\lambda}^m)(s,\theta)\Phi_{ss}(s)\,ds\right|\,d\theta\\
&\leq C_1\|\Phi_s\|_{\infty}|{\rm
supp}\Phi|\tau+\frac{C_2}{\lambda}\|\Phi_{ss}\|_{\infty}|{\rm
supp}\Phi|\tau=C(\Phi)\tau,
\end{split}
\end{equation*}
where by $|{\rm supp}\Phi|$ we understand the Lebesgue measure of
the (compact) support of $\Phi$. We have thus proved that
\begin{equation}\label{interm6}
\lim\limits_{\tau\to0}\int_{-\infty}^{\infty}w_{\lambda}(s,\tau)\Phi(s)\,ds=\int_{-\infty}^{\infty}w_{\lambda}(s,0)\Phi(s)\,ds,
\end{equation}
for any $\Phi\in\cd(\real)$ and $\lambda>0$, the convergence being
uniform with respect to $\lambda$ in any interval
$[\lambda_0,\infty)$.

It still remains to prove that
\begin{equation}\label{interm8}
\lim\limits_{\lambda\to\infty}\int_{-\infty}^{\infty}w_{\lambda}(s,0)\Phi(s)\,ds=K\int_{0}^{\infty}\Phi(s)\,ds,
\end{equation}
for any $\Phi\in\cd(\real)$. To this end, we calculate:
\begin{equation*}
\begin{split}
\int_{-\infty}^{\infty}w_{\lambda}(s,0)\Phi(s)\,ds&=\left(\int_{-\infty}^{0}+\int_0^{\infty}\right)w_{0}(\lambda
s)\Phi(s)\,ds=K\int_{0}^{\infty}\Phi(s)\,ds\\&+\int_{0}^{\infty}(w_0(\lambda
s)-K)\Phi(s)\,ds+\int_{-\infty}^{0}w_0(\lambda s)\Phi(s)\,ds.
\end{split}
\end{equation*}
Recall that $\lim\limits_{s\to\infty}w_0(s)=K$ and
$\lim\limits_{s\to-\infty}w_0(s)=0$. This implies
$$
\lim\limits_{\lambda\to\infty}w_0(\lambda s)\Phi(s)=0, \quad {\rm
for \ any \ } s<0,
$$
and
$$
\lim\limits_{\lambda\to\infty}(w_0(\lambda s)-K)\Phi(s)=0, \quad
{\rm for \ any \ } s>0,
$$
with pointwise convergence in both cases. Moreover, since $\Phi$ is
compactly supported and $w_0\in L^{\infty}(\real)$, we can apply the
Lebesgue's dominated convergence theorem to find
$$
\lim\limits_{\lambda\to\infty}\int_{-\infty}^{0}w_0(\lambda
s)\Phi(s)\,ds=\lim\limits_{\lambda\to\infty}\int_{0}^{\infty}(w_0(\lambda
s)-K)\Phi(s)\,ds=0,
$$
to conclude that \eqref{interm8} holds. Joining \eqref{interm6} and
\eqref{interm8}, we readily get \eqref{interm5}, as wanted. Thus,
$w_{\infty}$ is a generalized (entropy) solution for the
conservation law \eqref{CL} with initial condition $KH$. By
Kruzhkov's Theorem \cite{KR, Serre}, we find that $w_{\infty}\equiv
V$.

There is a last part in the four-step method, that is, rephrazing
the results in terms of the initial variables. We have just proved
that
$$
|w_{\lambda}(s,\tau)-w_{\infty}(s,\tau)|\to0 \quad {\rm as} \
\lambda\to\infty,
$$
uniformly in $(s,\tau)$ in compact subsets of
$\real\times[0,\infty)$. We put $\tau=1$, then we relabel
$\lambda=\tau$, to get that
$$
|w(s\tau,\tau)-w_{\infty}(s,1)|\to0 \quad {\rm as} \ \tau\to\infty,
$$
uniformly for $s$ in compact sets of $\real$. Choosing compacts of
the type $[-R,R]$ for $R>0$ large, this is equivalent to say that
$$
|w(s,\tau)-V(s,\tau)|\to0 \quad {\rm as} \ \tau\to\infty,
$$
uniformly for $s\in[-R\tau,R\tau]$. By undoing transformation
\eqref{tr1} and going back to the initial variables $(x,t)$, we
obtain \eqref{conv.unif} in sets of the type $\{e^{-Rt}\leq|x|\leq
e^{Rt}\}$, for any $R>0$.

\medskip

\noindent \textbf{Step 5. Behavior at the origin.} We go back to
initial variables and show that $u(0,t)=K$ for any $t>0$. Assume, by
contradiction, that there exists $t_0>0$ such that $u(0,t_0)=K_1<K$
(if $K_1>K$, things are completely similar). Then, we can start the
evolution taking $t=t_0$ as initial time; by uniqueness, the
solution to the Cauchy problem with $v_0(x)=u(x,t_0)$ will be
$v(x,t)=u(x,t+t_0)$. Applying \eqref{conv.unif} for this $v$, we
find that
$$
\lim\limits_{t\to\infty}|u(x,t+t_0)-E_{K_1}(x,t)|=0,
$$
uniformly in any set of the form $\{e^{-Rt}\leq|x|\leq e^{Rt}\}$ for
any $R>0$. In particular, choosing $R$ sufficiently large, (the
precise condition is $R>mK^{m-1}(N-2)$), we reach a contradiction,
as in the set $\{e^{-Rt}\leq|x|\leq e^{Rt}\}$, the two profiles
$E_{K}$ and $E_{K_1}$ are essentially different:
$\lim\limits_{t\to\infty}\|E_{K}(t)-E_{K_1}(t)\|_{\infty}=K-K_1>0$.
Hence $u(0,t)=K$ for any $t>0$.

\medskip

\noindent \textbf{Step 6. Uniform convergence in the whole space.}
We have shown up to now that \eqref{conv.unif} holds true uniformly
in sets of the form
$$
\{e^{-Rt}\leq|x|\leq e^{Rt}\}, \quad {\rm for \ any \ } R>0.
$$
In order to extend the uniform convergence to the whole $\real^N$,
we essentially use the fact that $u(t)$ is nonincreasing for any
$t>0$. Let some $\e_0>0$ fixed. Then, for any $t>0$, we have
$u(\e_0,t)\leq u(x,t)=u(|x|,t)\leq u(0,t)=K$, that is, $u(\cdot,t)$
is uniformly Cauchy in $[0,\e_0]$, whence the uniform convergence is
extended up to the origin. A similar argument holds for the tail
part $\{|x|\geq e^{Rt}\}$ closing the proof.
\end{proof}

\section{Asymptotic convergence for general
solutions}\label{sec.general}

We are now ready to prove our results for non-radially symmetric
solutions, which are Theorem \ref{th.2b} and Corollary
\ref{cor.decay}.
\begin{proof}[Proof of Theorem \ref{th.2b}]
Let $u$ be a solution to \eqref{eq1} with initial condition $u_0$
satisfying \eqref{initdata}, \eqref{initdata2} and
\eqref{initcond3}, with $u_0(0)=K>0$. We define the following
radially symmetric functions:
\begin{equation}\label{evolvents}
u_0^{\pm}:\real^N\mapsto[0,\infty), \quad u_0^{-}(r)=\inf\{u_0(x):
|x|=r\}, \quad u_0^{+}(r)=\sup\{u_0(x): |x|=r\}, \quad r=|x|.
\end{equation}
It is obvious that $u_0^{-}(x)\leq u_0(x)\leq u_0^{+}(x)$, for any
$x\in\real^N$, and $u_0^{-}(0)=u_0^{+}(0)=K$. Moreover, both
$u_0^{-}$ and $u_0^{+}$ are continuous at $x=0$ and belong to
$L^{\infty}(\real^N)$, as $0< \|u_0^{-}\|_{\infty}=
\|u_0\|_{\infty}=\|u_0^{+}\|_{\infty}=K$.

It remains to check that $u_0^{-},u_0^{+}\in L^1_2(\real^N)$. We
only have to check this in sets that are close to $x=0$ and to
infinity. Since $u_0(0)=K$ and it is continuous, there exists $r>0$
such that $K-1<u_0(x)<K+1$, for any $x\in B(0,r)$. By definition, we
also have $K-1<u_0^{-}(x)\leq u_0^{+}(x)<K+1$, for any $x\in
B(0,r)$, hence
$$
0<\int_{B(0,r)}|x|^{-2}u_0^{-}(x)\,dx\leq\int_{B(0,r)}|x|^{-2}u_0^{+}(x)\,dx\leq
(K+1)\int_{B(0,r)}|x|^{-2}\,dx<\infty,
$$
as we are in dimension $N\geq3$. Concerning sets that are "close to
infinity", we deduce from \eqref{initcond3} that
$$
u_0^{-}(x),u_0^{+}(x)\leq|x|^{2-N-\delta}, \quad {\rm for \ any \ }
x\in\real^N\setminus B(0,R),
$$
whence
$$
0<\int_{\real^N\setminus
B(0,R)}|x|^{-2}u_0^{-}(x)\,dx\leq\int_{\real^N\setminus
B(0,R)}|x|^{-2}u_0^{+}(x)\,dx\leq\int_{\real^N\setminus
B(0,R)}|x|^{-N-\delta}\,dx<\infty.
$$
This, together with the uniform boundedness in the compact set
$\overline{B(0,R)}\setminus B(0,r)$, show that $u_0^{-},u_0^{+}\in
L^1_2(\real^N)$. By Theorem \ref{th.wp}, there exist $u^{-},u^{+}$
solutions to Eq. \eqref{eq1} with initial data $u_0^{-},u_0^{+}$
respectively; as Eq. \eqref{eq1} is rotationally invariant, $u^{-}$
and $u^+$ are radially symmetric and they fulfill the assumptions in
Theorem \ref{th.2}. It follows that
$$
\lim\limits_{t\to\infty}|u^{-}(x,t)-E_{K}(x,t)|=\lim\limits_{t\to\infty}|u^{+}(x,t)-E_{K}(x,t)|=0,
$$
uniformly in $\real^N$. By standard comparison, $u^{-}(x,t)\leq
u(x,t)\leq u^{+}(x,t)$ for any $(x,t)\in\real^N\times[0,\infty)$,
and the conclusion of Theorem \ref{th.2b} follows in an obvious way.
Proof of Theorem \ref{th.2b} under conditions in Remark \ref{rmk} is
totally similar.
\end{proof}

\begin{proof}[Proof of Corollary \ref{cor.decay}]
Define $u_0^{\pm}$ and $u^{\pm}$ as above. Then $u^{+}$ and $u^{-}$
are radially symmetric solutions satisfying the conditions in
Theorem \ref{th.1}, and $u^{-}$ is nontrivial due to the condition
\eqref{inf0}. By Theorem \ref{th.1b}, $\|u^{\pm}(t)\|_{\infty}$
decay exactly with rate $t^{-1/m}$. Since by comparison,
$u^{-}(x,t)\leq u(x,t)\leq u^{+}(x,t)$ for any
$(x,t)\in\real^N\times[0,\infty)$, the same decay rate holds true
for $\|u(t)\|_{\infty}$.
\end{proof}
Notice that we cannot have a more precise asymptotic convergence
result for general solutions $u$ to Eq. \eqref{eq1} through this
method, since the limit profiles depend essentially on the mass
$M_{u_0}$, and there is no obvious connection between $M_{u_0}$ and
$M_{u_0^{\pm}}$.

\section{Further results, extensions and
applications}\label{sec.final}

In this final section, we gather some facts that, at a formal level,
extend or apply our analysis to other equations or different
conditions.

\subsection{A different critical density: $\gamma=N-(N-2)/m$}

In a previous work \cite[Theorem 2.1]{IRS}, it has been noticed that
Eq. \eqref{eq1} can be mapped, at the level of radially symmetric
solutions, into a similar equation presenting a different, but also
critical, density. More precisely, at a formal level, given a
radially symmetric solution $u$ to Eq. \eqref{eq1}, we define
\begin{equation}\label{tr2}
\tilde{u}(r,t):=r^{(2-N)/m}u(r^{-1},t), \quad r=|x|, \end{equation}
and find (\cite{IRS}) that $\tilde{u}$ is a radially symmetric
solution to
\begin{equation}\label{gama2}
|x|^{-\gamma_2}\tilde{u}_t=\Delta\tilde{u}^m, \quad
\gamma_2=N-\frac{N-2}{m}.
\end{equation}
The critical behavior of the exponent $\gamma_2$ has been analyzed
in previous works as \cite{KRV10}, \cite[Subsection 3.3]{IRS}. In
particular, the theory developed in \cite[Section 6]{KRV10} holds
exactly for $\gamma<\gamma_2$ but fails to hold for this borderline
case.

Taking into account \eqref{tr2} and our results, and refraining from
performing a rigorous analysis of Eq. \eqref{gama2}, we can give
some ideas about what is expected to happen with its solutions. The
profiles $E_K$ and $F$ transform into
\begin{equation*}
\tilde{E}_{K}(x,t)=\left\{\begin{array}{ll}0,& {\rm if} \ |x|\leq1,\\
|x|^{(2-N)/m}\left[\frac{1}{m(N-2)}\frac{\log|x|}{t}\right]_{+}^{1/(m-1)},&
{\rm if} \ 1<|x|<e^{m(N-2)K^{m-1}t},\\ |x|^{(2-N)/m}K, & {\rm if} \
|x|\geq e^{m(N-2)K^{m-1}t},\end{array}\right.
\end{equation*}
and
\begin{equation*}
\tilde{F}(x,t)=\left\{\begin{array}{ll}t^{-1/m}|x|^{(2-N)/m}\left[\frac{\log|x|t^{-1/m}}{m(N-2)}\right]_{+}^{1/(m-1)},
& {\rm if} \ 0\leq|x|<e^{kt^{1/m}},\\0, & {\rm if} \ |x|\geq
e^{kt^{1/m}},\end{array}\right.
\end{equation*}
the first being supported in the interval $[1,\infty)$ and the
second in $[1,e^{kt^{1/m}}]$.

Moreover, there exists an explicit family of self-similar solutions
to \eqref{gama2}, that is obtained starting from the logarithmic
Barenblatt solutions to \eqref{eq1} given in \eqref{Bar}; by mapping
them via \eqref{tr2}, we deduce explicit solutions to \eqref{gama2}
having the self-similar form
\begin{equation}\label{Bar2}
\tilde{B}_{D}(x,t)=t^{-1/m}U_D(|x|t^{1/(m-1)(N-2)}), \quad
U_D(\xi)=\xi^{(2-N)/m}\left[D+\frac{1}{m(N-2)}\log\xi\right]_{+}^{1/(m-1)},
\end{equation}
for any $D>0$. Up to our knowledge, the self-similar functions in
\eqref{Bar2} are new. Notice that they have a time decay $t^{-1/m}$,
but also a backward evolution of the support, that is,
$$
{\rm supp}\tilde{B}_D=\left\{(x,t)\in\real^N\times(0,\infty):
|x|>e^{-Dm(N-2)}t^{-1/(m-1)(N-2)}\right\}.
$$
By applying transformation \eqref{tr2} to our results, it is
expected then that $\tilde{E}_K$ be the family of general asymptotic
profiles of solutions $\tilde{u}$ to \eqref{gama2} decaying at
infinity exactly like $K|x|^{(2-N)/m}$, with
$\lim\limits_{|x|\to0}\tilde{u}(x,t)=0$ while $\tilde{F}$ is the
general asymptotic profile of solutions $\tilde{u}$ to \eqref{gama2}
having the same behavior at the origin but a weaker decay as
$|x|\to\infty$.

Since \eqref{tr2} is an inversion, a stronger role will play our new
explicit solution in \eqref{Bar2}, with respect to large-time
behavior in \emph{inner regions}, close to the origin. Indeed, by
rephrazing Theorem \ref{th.3} for $\tilde{u}$ and performing the
changes to self-similar variables as in \eqref{Bar2}, one gets the
following expected asymptotic convergence for $\gamma=\gamma_2$:
$$
\lim\limits_{t\to\infty}|\tilde{u}(x,t)-\tilde{B}_D(x,t)|=0,
$$
uniformly in sets of the form $|x|\leq Kt^{-1/(m-1)(N-2)}$ for any
$K>0$, that is, in small \emph{inner sets} shrinking to the origin
as $t\to\infty$. This result will be interesting from the point of
view of explaining the influence of the singularity, which is the
most important feature of \eqref{gama2}.

A rigorous study of this critical case will be left for future work.

\subsection{Some open questions}

Related to the analysis performed in the present work, we leave
below a list of, in our opinion, interesting questions that can be
addressed in future developments of the subject.

\medskip

\noindent \textbf{1. The inner behavior for nonsingular densities.}
As explained in the Introduction, there was an important
mathematical interest for the study of \eqref{eq2} with densities
$\varrho$ that are regular at $x=0$ and have $\varrho(x)\leq
C|x|^{-\gamma}$ as $|x|\to\infty$, see \cite{RV06, RV08, RV09,
KRV10, NR}. In particular, it has been shown that the large-time
behavior is strongly related to fundamental solutions to our Eq.
\eqref{eq1}. The density $\gamma=2$ is critical \cite{KRV10} and the
authors of \cite{NR} prove the large-time behavior in outer sets,
similar to our Theorem \ref{th.3}. The problem of understanding the
\emph{inner behavior} is left open. We don't know whether our
results in the present paper can give more light on this subject, as
we analyze the inner behavior for Eq. \eqref{eq1}, but the
difference between densities $\varrho$ as considered in
\cite[Subsection 2.1]{NR} and our singular $\varrho(x)=|x|^{-2}$
might play an essential role.

\medskip

\noindent \textbf{2. Removing a condition in Theorems \ref{th.2} and
\ref{th.2b}.} One would like to eliminate the condition $0\leq
u_0(x)\leq K$ from the statement of Theorems \ref{th.2}, \ref{th.2b}
and prove them for general $u_0\in L^{\infty}(\real^N)$; this would
allow functions that can have a peak in $\{0<|x|<\infty\}$ but still
satisfying
$$
\lim\limits_{|x|\to\infty}u_0(x)=0, \quad u_0(0)=K>0.
$$
We conjecture that Theorem \ref{th.2} and Theorem \ref{th.2b} remain
true in this case, that is, there exists some mechanism in the Eq.
\eqref{eq1} forcing a maximum attained at some point different from
the origin to go down in time. This is confirmed by the numerical
experiments, as Figure \ref{figure3} suggests.

\medskip

\begin{figure}[ht!]
  \begin{center}
  \includegraphics[width=12cm,height=8cm]{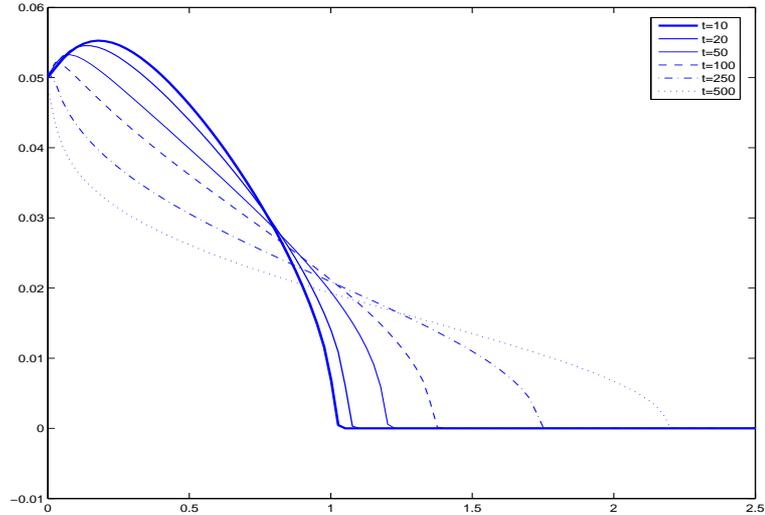}
  \end{center}
  \caption{Evolution of a solution with maximum attained outside the origin.}\label{figure3}
\end{figure}

\medskip

We can see that starting from a radially symmetric initial data with
a maximum attained for some $|x|=r(0)\in(0,\infty)$, and letting for
any $t>0$
$$
u(r(t),t)=\max\{u(|x|,t):x\in\real^N\},
$$
then $u(r(t),t)$ decreases with $t$, but at the same time
$r(t)\to0$. Thus, the slope of the graph in $[0,r(t)]$ may be quite
big for $t>0$ very large, which shows that, even if we expect the
same result as in Theorem \ref{th.2} to hold true, our technique
essentially based on a boundedness of the derivative cannot be used.
It seems that we need some different ideas.

\medskip

\noindent \textbf{3. Large time behavior for general solutions in
Theorem \ref{th.1b}.} While for solutions to \eqref{eq1} with
initial data $u_0(0)=K>0$, our analysis holds true for general (not
necessarily radially symmetric) solutions, this problem remains open
in the case when $u_0(0)=0$. Notice that a similar proof as that of
Theorem \ref{th.2b} does not hold, due to the essential dependence
of the limit profile on some weighted mass of the initial condition.
As suggested in \cite{IS12}, one might expect a negative answer.

\medskip

\textsc{Acknowledgements.} R. I. partially supported by the Spanish
project MTM2012-31103. A. S. partially supported by the Spanish
project MTM2011-25287. Part of this work has been done during visits
by A. S. to the Departamento de Análisis Matemático of the
Universidad de Valencia.

\bibliographystyle{plain}

\end{document}